%% file: main.tex
\documentclass[runningheads, envcountsame, a4paper]{llncs}
\usepackage{epstopdf}
\usepackage[hyphens]{url}
\usepackage{microtype}
\usepackage{graphicx}
\usepackage{amsmath}
\usepackage{amsfonts}
\usepackage{latexsym}
\usepackage{amssymb}
\usepackage{stmaryrd}              
\usepackage{tikz}
\usetikzlibrary{automata,positioning}

\usepackage{todonotes}

\newcommand{\N}{\mathbb{N}}
\newcommand{\Z}{\mathbb{Z}}

\newcommand{\Acal}{\mathcal{A}}
\newcommand{\Bcal}{\mathcal{B}}
\newcommand{\Ccal}{\mathcal{C}}
\newcommand{\Fcal}{\mathcal{F}}
\newcommand{\Gcal}{\mathcal{G}}
\newcommand{\Hcal}{\mathcal{H}}
\newcommand{\Lcal}{\mathcal{L}}
\newcommand{\Qcal}{\mathcal{Q}}
\newcommand{\Scal}{\mathcal{S}}

\newcommand{\Ucal}{\mathcal{U}}
\newcommand{\Xcal}{\mathcal{X}}
\newcommand{\Zcal}{\mathcal{Z}}

\newcommand{\be}{{\boldsymbol{e}}}
\newcommand{\bk}{{\boldsymbol{k}}}
\newcommand{\bell}{{\boldsymbol{\ell}}}
\newcommand{\bm}{{\boldsymbol{m}}}
\newcommand{\bn}{{\boldsymbol{n}}}

\DeclareMathOperator{\rep}{rep}
\DeclareMathOperator{\pad}{pad}

\DeclareMathOperator{\val}{val}
\DeclareMathOperator{\width}{width}
\DeclareMathOperator{\height}{height}

\newcommand{\Zrange}[1]{\llbracket0,#1\rrbracket}

\title{A Numeration System for Fibonacci-like\\Wang Shifts}

\author{S\'ebastien Labb\'e\inst{1} \and Jana Lep\v{s}ov\'a\inst{1,2}}
\authorrunning{S. Labb\'e, J. Lep\v{s}ov\'a}   

\institute{Univ. Bordeaux, CNRS,  Bordeaux INP, LaBRI, UMR 5800, F-33400,
        Talence, France\\
\and
FNSPE, CTU in Prague,
Trojanova 13, 120 00 Praha, Czech Republic\\
\email{sebastien.labbe@labri.fr},
\email{jana.lepsova@labri.fr}
}

\begin{document}

\maketitle

\begin{abstract}
Motivated by the study of Fibonacci-like Wang shifts,
we define a numeration system for $\mathbb{Z}$ and $\mathbb{Z}^2$
based on the binary alphabet $\{0,1\}$.
We introduce a set of 16 Wang tiles 
that admits a valid tiling of the plane described by
a deterministic finite automaton
taking as input the representation of a position $(m,n)\in\mathbb{Z}^2$ and
outputting a Wang tile.
\end{abstract}

\date{\today}

\section{Introduction}
\label{sec:intro}

A theorem of Cobham (1972) says that 
a sequence $u=(u_n)_{n\geq0}$ is $k$-automatic
with $k\geq2$
if and only if it is the image,
under a coding, of a fixed point of a $k$-uniform morphism 
\cite[Theorem 6.3.2]{MR1997038}.
This result was extended to non-uniform morphisms \cite{zbMATH01916667},
see also \cite[Theorem 3.4.1]{zbMATH05707089},
by replacing the usual base-$k$ expansion of nonnegative integers by an
abstract numeration system and a regular language \cite{zbMATH01588127}.
It was later extended to configurations $x:\N^d\to\Sigma$
in dimension $d\geq1$ based on the notion of shape-symmetric morphic words
\cite{zbMATH05701985},
see also \cite[Theorem 3.4.26]{zbMATH05707089} and \cite[\S~5]{zbMATH07258478}.

In this article, we explore an extension of Cobham's result beyond the nonnegative octant $\N^d$ 
to include configurations $x:\Z^d\to\Sigma$ defined on the whole lattice $\Z^d$.
We concentrate on one example in dimension $d=2$.
The example is motivated by the study of Wang tilings of the plane.
Given an alphabet $\Ccal$ of colors,
a Wang tile is a 4-tuple $(a,b,c,d)\in\Ccal^4$ that
represents the labeling of the edges of a unit square, by convention, in the order 
corresponding to a positive rotation on the complex plane, i.e., 
$a$ is the east edge label, 
$b$ is the north edge label, 
$c$ is the west edge label and
$d$ is the south edge label.
\begin{figure}[h]
	\centering
	\includegraphics[width=0.6\linewidth]{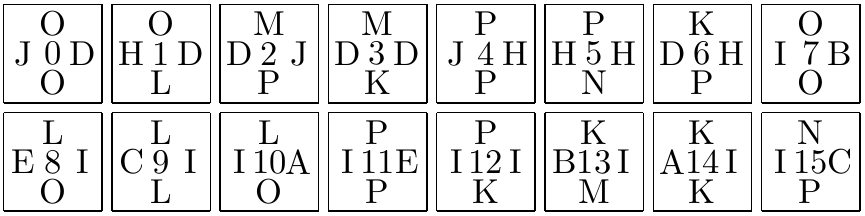}
	\caption{The set $\Zcal=\{z_0,\dots,z_{15}\}$ of 16 Wang tiles. 
    The index $i$ of the tile $z_i$ is written in the center of each tile.}
	\label{fig:16tiles}
\end{figure} 

We introduce a set $\Zcal=\{z_0,\dots,z_{15}\}$ of 16 Wang tiles shown in Figure~\ref{fig:16tiles}.
The set $\Zcal$ is a simplification of an existing aperiodic set of 19 Wang tiles
\cite{MR3978536} after identification of few colors, which was shown to be
related \cite{MR4226493} to the smallest set of aperiodic Wang tiles found by
Jeandel and Rao \cite{Jeandel2021}.

A valid configuration over the set of Wang tiles $\Zcal$ is a function $f:\Z^2\to\{0,\dots,15\}$ such that adjacent tiles have the same
label on their common edge, i.e., 
for every $n\in\Z^2$,
the east label of the tile $z_{f(n)}$ is equal to the west label of the tile $z_{f(n+e_1)}$
and
the north label of the tile $z_{f(n)}$ is equal to the south label of the tile $z_{f(n+e_2)}$.
    A partial valid configuration is shown in 
Figure~\ref{fig:configuration}. It is this particular configuration that is 
linked with a numeration system in Theorem 1.
The set $\Omega_\Zcal$ of valid configurations $f:\Z^2\to\{0,\dots,15\}$ is
called the Wang shift associated to the set of Wang tiles~$\Zcal$.

\begin{figure}
	\centering
	\includegraphics[width=0.55\linewidth]{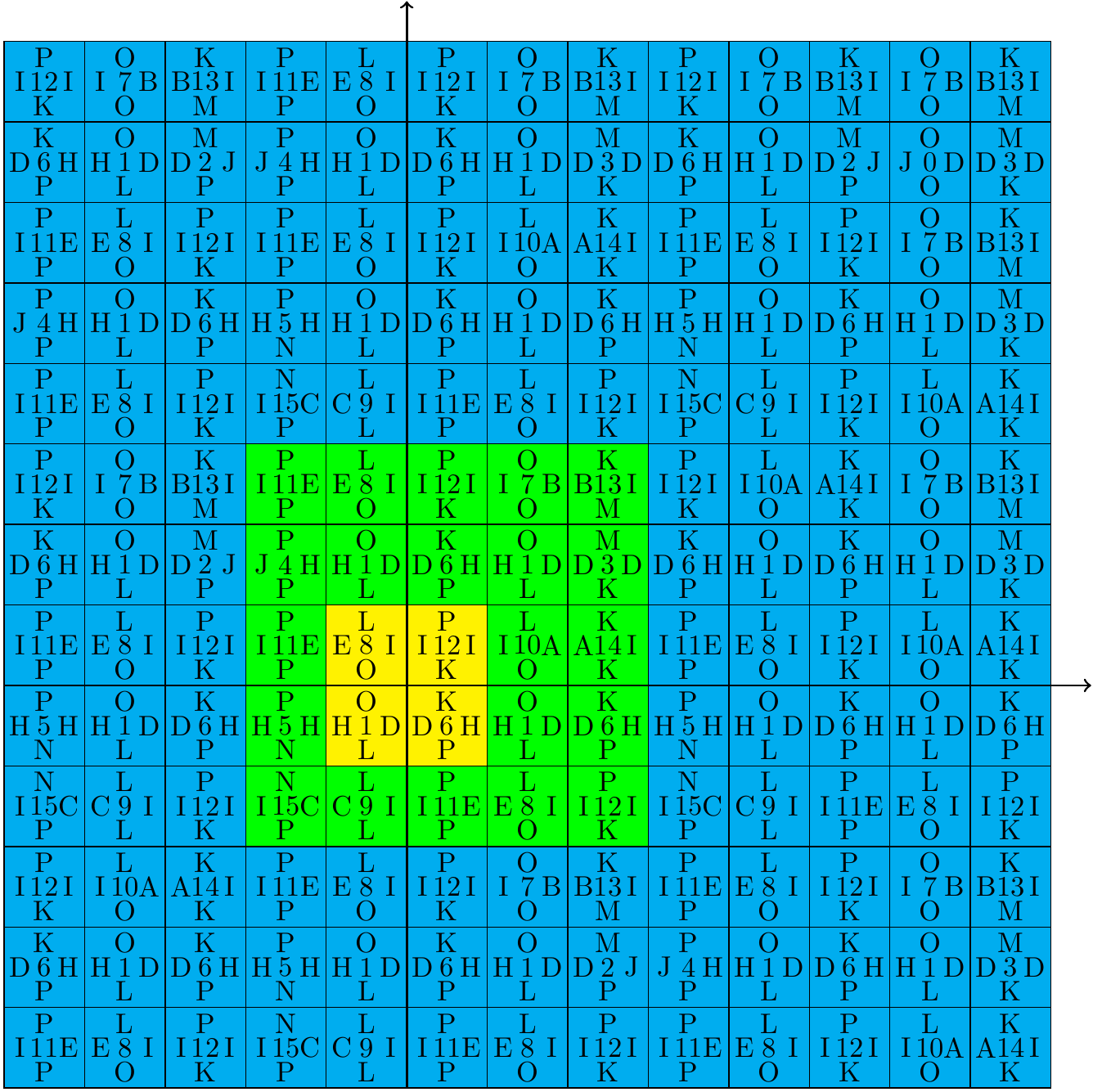}
    \caption{A partial valid configuration $[-5,8[^2\to\{0,\dots,15\}$ with
    the set $\Zcal$ of Wang tiles.}
	\label{fig:configuration}
\end{figure}

We show the following result which states a link between a specific 
configuration over
the set $\Zcal$ of Wang tiles shown in Figure~\ref{fig:16tiles},
a numeration system for $\Z^2$
and a deterministic finite automaton with output 
(DFAO). 
Definition of DFAO is recalled in Section 3
and corresponds to the classic definition \cite{MR1997038}.

\begin{theorem}\label{theoremA}
    Let $\Zcal$ be the set of 16 Wang tiles shown in Figure~\ref{fig:16tiles}.
    There exist a valid Wang configuration $x\in\Omega_\Zcal$
    and a DFAO $\Acal$
    and a numeration system $\Fcal$ for $\Z^2$ with a representation function 
    $\rep_\Fcal:\Z^2\to\{
        \left(\begin{smallmatrix}0\\0\end{smallmatrix}\right),
        \left(\begin{smallmatrix}0\\1\end{smallmatrix}\right),
        \left(\begin{smallmatrix}1\\0\end{smallmatrix}\right),
        \left(\begin{smallmatrix}1\\1\end{smallmatrix}\right)
        \}^*$
	such that the tile at position $\bn\in\Z^2$ in $x$
	is $x_\bn=\Acal(\rep_\Fcal(\bn))$.
\end{theorem}

In fact, the Wang shift $\Omega_\Zcal$ is self-similar, minimal and aperiodic.
Moreover, it is topologically conjugate to the Wang shift $\Omega_\Ucal$
generated by the set $\Ucal$ of 19 Wang tiles introduced in \cite{MR3978536}.
These results are not shown here due to lack of space and will be part of an
extended version of this article.

The article is structured as follows.
In Section~\ref{sec:fibo-num-system},
we introduce a Fibonacci numeration system for $\Z$ and $\Z^2$.
In Section~\ref{sec:the-automaton-for-Z},
we illustrate how to change the usual automaton in Cobham's theorem for $\N$ so
that it can read the representation of integers including the negative ones.
In Section~\ref{sec:two-dim-words},
we recall the definitions and notations for
two-dimensional words, languages and morphisms.
The self-similarity of the Wang shift $\Omega_\Zcal$
is stated in Section~\ref{sec:self-similarity} (the proof is available in the
appendix of the preprint version) from which the automaton of
Theorem~\ref{theoremA} 
is deduced, see Figure~\ref{fig:automaton}.
The proof of Theorem~\ref{theoremA} is done in
Section~\ref{sec:proof-main-results}.

\textbf{Acknowledgements}.
This work was supported by the Agence Nationale de la Recherche through the
project Codys (ANR-18-CE40-0007).

\section{A Fibonacci Numeration System for $\Z$ and $\Z^2$}
\label{sec:fibo-num-system}



Let $(F_n)_{n\geq 1}$ be the Fibonacci sequence
defined with the reccurent relation
$$ F_0 = 1, F_1 = 1, F_2 = 2, F_{n+2} = F_{n+1} + F_n \ \ \text{ for all } n \geq 1.$$ 
By the Zeckendorf theorem \cite{MR308032} every nonnegative integer $n$ can be represented
as a unique sum of nonconsecutive Fibonacci numbers $ n = \sum_{i=1}^{\ell}w_i
F_i,$ where 
$\ell = \max\left\{i\in\N_0 : F_i \leq n\right\}$, 
$w_i \in \left\{0,1\right\}$ 
and $w_i  w_{i+1} = 0, \text{ for all } i \in \left\{1,2,..., \ell-1\right\}$.

Inspired by the Two's complement, ``\emph{the most common method of representing signed
integers on computers}'',\footnote{\url{https://en.wikipedia.org/wiki/Two's_complement}} we introduce a numeration system $\Fcal$ which extends
the Fibonacci numeration system to all $n\in\Z$ as follows.
For each binary word $w=w_{2k+1}w_{2k}\cdots w_1\in\Sigma^{2k+1}$ of odd
length over the alphabet $\Sigma=\{0,1\}$, we define
\[
    \val_\Fcal(w) = \sum_{i=1}^{2k}w_i F_i - w_{2k+1}F_{2k}.
\]
The following lemma is an exercise based on the Fibonacci recurrence.
\begin{lemma}
    Let $k\in\N$ and
    $w\in\Sigma^{2k}\setminus \Sigma^*11\Sigma^*$.
    We have
    \begin{enumerate}
\item $\val_\Fcal(0w)=\val_\Fcal(000w)=\val_\Fcal(110w)$,
\item $\val_\Fcal(1w)=\val_\Fcal(101w)$,
\item $\val_\Fcal(100w)=\val_\Fcal(000w)-F_{2k+2}$,
\item $0\leq\val_\Fcal(0w)<F_{2k+1}$,
\item $-F_{2k+2}\leq\val_\Fcal(100w)<0$.\qed
    \end{enumerate}
\end{lemma}

Thus, the first digit of $w\in\Sigma^{2k+1}\setminus \Sigma^*11\Sigma^*$ gives the sign (nonnegative or negative) of the value $\val_\Fcal(w)$.
We can show the following.
\begin{proposition}
For every $n\in\Z$ there exists a
    unique odd-length word 
    \[w\in\Sigma(\Sigma\Sigma)^*\setminus
    \left( \Sigma^*11\Sigma^* \cup 000\Sigma^* \cup 101\Sigma^*\right)\]
    such that $n=\val_\Fcal(w)$.
\end{proposition}
\begin{proof}
    (Unicity).
    Let $w,w'\in\Sigma(\Sigma\Sigma)^*\setminus
    \left( \Sigma^*11\Sigma^* \cup 000\Sigma^* \cup 101\Sigma^*\right)$
    of minimal length
    such that $\val_\Fcal(w)=\val_\Fcal(w')$.
    If $w\in1\Sigma^*$, then $\val_\Fcal(w)=\val_\Fcal(w')<0$
    and $w'\in1\Sigma^*$ as well.
    In fact, we must have $w,w'\in100\Sigma^*$.
    Thus $w=10u$ and $w=10u'$ for some words $u,u'$
    such that $\val_\Fcal(u)=\val_\Fcal(u')$. This contradicts the minimality
    of the lengths of $w$ and $w'$.
    If $w\in0\Sigma^*$, then $\val_\Fcal(w)=\val_\Fcal(w')\geq0$
    and $w'\in0\Sigma^*$ as well.
    But $w,w'\notin000\Sigma^*$, thus 
    $w\in\{01,001\}u$
    and
    $w'\in\{01,001\}u'$ for some $u,u'\in\Sigma^*$.
    From Zeckendorf's theorem applied to $1u$ and $1u'$, we conclude that $u=u'$.

    (Existence).
    If $n=0$, then $n=0=\rep_\Fcal(0)$.
    Assume that $n>0$.
    From Zeckendorf's theorem, there exists
    a unique $u\in 1\Sigma^*\setminus \Sigma^*11\Sigma^*$ 
    such that $n=\val_F(u)$.
    If $u$ has odd-length, then $n=\val_\Fcal(00u)$.
    If $u$ has even-length, then $n=\val_\Fcal(0u)$.
    Now assume that $n<0$.
    Let $k\in\N$ be such that
    $-F_{2k}\leq n<-F_{2k-2}$.
    We have
    $0\leq n+F_{2k}<F_{2k}-F_{2k-2}=F_{2k-1}$.
    Let $w\in\Sigma^{2k-2}\setminus \Sigma^*11\Sigma^*$ such that
    $\val_\Fcal(000w)=n+F_{2k}$.
    We thus have $n=\val_\Fcal(100w)$.
    \qed
\end{proof}

\begin{definition}[Numeration system $\Fcal$ for $\Z$]
For each $n\in\Z$, we denote by $\rep_\Fcal(n)$ the unique word satisfying the
proposition. 
\end{definition}

The numeration system $\Fcal$ is illustrated in
Figure~\ref{fig:numeration-system-F}.

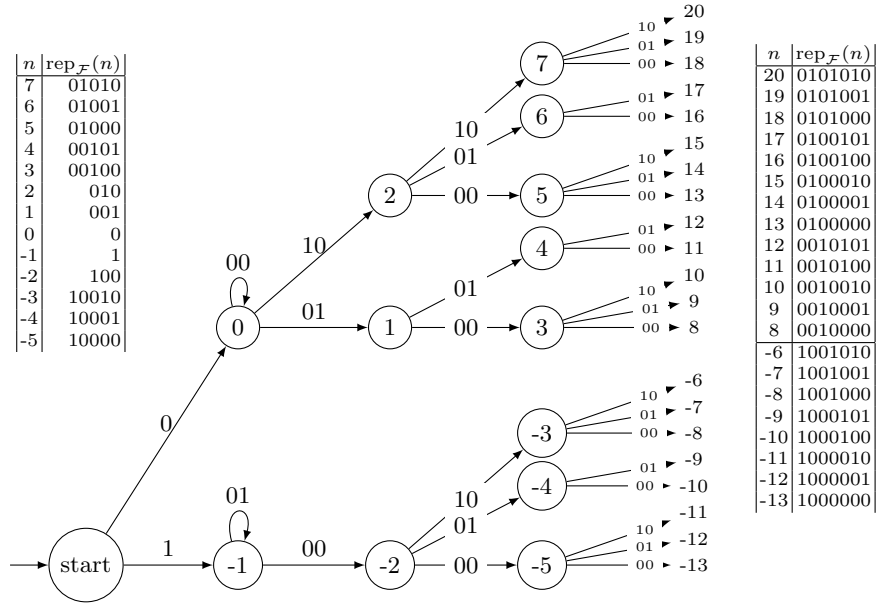
\begin{figure}[h]
\begin{center}
    \begin{tikzpicture}[>=latex]
    \begin{scope}[yscale=.35]
    \node (S)  [draw,circle] at (0,-8) {start};
    \node (0)  [draw,circle] at (2,1) {0};
    \node (-1) [draw,circle] at (2,-8) {-1};
    \node (1)  [draw,circle] at (4,1) {1};
    \node (2)  [draw,circle] at (4,6) {2};
    \node (-2) [draw,circle] at (4,-8) {-2};
    \node (3)  [draw,circle] at (6,1)  {3};
    \node (4)  [draw,circle] at (6,4)  {4};
    \node (5)  [draw,circle] at (6,6)  {5};
    \node (6)  [draw,circle] at (6,9)  {6};
    \node (7)  [draw,circle] at (6,11) {7};
    \node (-3) [draw,circle] at (6,-3) {-3};
    \node (-4) [draw,circle] at (6,-5) {-4};
    \node (-5) [draw,circle] at (6,-8) {-5};
    \node (8)    at (8,1)  {\scriptsize 8};
    \node (9)    at (8,2)  {\scriptsize 9};
    \node (10)   at (8,3)  {\scriptsize 10};
    \node (11)   at (8,4)  {\scriptsize 11};
    \node (12)   at (8,5)  {\scriptsize 12};
    \node (13)   at (8,6)  {\scriptsize 13};
    \node (14)   at (8,7)  {\scriptsize 14};
    \node (15)   at (8,8)  {\scriptsize 15};
    \node (16)   at (8,9)  {\scriptsize 16};
    \node (17)   at (8,10) {\scriptsize 17};
    \node (18)   at (8,11) {\scriptsize 18};
    \node (19)   at (8,12) {\scriptsize 19};
    \node (20)   at (8,13) {\scriptsize 20};
    \node (-6)   at (8,-1) {\scriptsize -6};
    \node (-7)   at (8,-2) {\scriptsize -7};
    \node (-8)   at (8,-3) {\scriptsize -8};
    \node (-9)   at (8,-4) {\scriptsize -9};
    \node (-10)  at (8,-5) {\scriptsize -10};
    \node (-11)  at (8,-6) {\scriptsize -11};
    \node (-12)  at (8,-7) {\scriptsize -12};
    \node (-13)  at (8,-8) {\scriptsize -13};
    \end{scope}
    \draw[loop above,->] (0) to node{00} (0);
    \draw[loop above,->] (-1) to node{01} (-1);
    \draw[<-] (S)  -- ++ (-1,0);
    \draw[->] (S)  -- node[above]{0} (0);
    \draw[->] (S)  -- node[above]{1} (-1);
    \draw[->] (0)  -- node[above]{10} (2);
    \draw[->] (0)  -- node[above]{01} (1);
    \draw[->] (-1) -- node[above]{00} (-2);
    \draw[->] (2)  -- node[fill=white]{10} (7);
    \draw[->] (2)  -- node[fill=white]{01} (6);
    \draw[->] (2)  -- node[fill=white]{00} (5);
    \draw[->] (1)  -- node[fill=white]{01} (4);
    \draw[->] (1)  -- node[fill=white]{00} (3);
    \draw[->] (-2) -- node[fill=white]{10} (-3);
    \draw[->] (-2) -- node[fill=white]{01} (-4);
    \draw[->] (-2) -- node[fill=white]{00} (-5);
    \draw[->] (7)  -- node[fill=white,near end]{\tiny 10} (20);
    \draw[->] (7)  -- node[fill=white,near end]{\tiny 01} (19);
    \draw[->] (7)  -- node[fill=white,near end]{\tiny 00} (18);
    \draw[->] (6)  -- node[fill=white,near end]{\tiny 01} (17);
    \draw[->] (6)  -- node[fill=white,near end]{\tiny 00} (16);
    \draw[->] (5)  -- node[fill=white,near end]{\tiny 10} (15);
    \draw[->] (5)  -- node[fill=white,near end]{\tiny 01} (14);
    \draw[->] (5)  -- node[fill=white,near end]{\tiny 00} (13);
    \draw[->] (4)  -- node[fill=white,near end]{\tiny 01} (12);
    \draw[->] (4)  -- node[fill=white,near end]{\tiny 00} (11);
    \draw[->] (3)  -- node[fill=white,near end]{\tiny 10} (10);
    \draw[->] (3)  -- node[fill=white,near end]{\tiny 01}  (9);
    \draw[->] (3)  -- node[fill=white,near end]{\tiny 00}  (8);
    \draw[->] (-3) -- node[fill=white,near end]{\tiny 10} (-6);
    \draw[->] (-3) -- node[fill=white,near end]{\tiny 01} (-7);
    \draw[->] (-3) -- node[fill=white,near end]{\tiny 00} (-8);
    \draw[->] (-4) -- node[fill=white,near end]{\tiny 01} (-9);
    \draw[->] (-4) -- node[fill=white,near end]{\tiny 00} (-10);
    \draw[->] (-5) -- node[fill=white,near end]{\tiny 10} (-11);
    \draw[->] (-5) -- node[fill=white,near end]{\tiny 01} (-12);
    \draw[->] (-5) -- node[fill=white,near end]{\tiny 00} (-13);
\node at (-.2,2.0) {\scriptsize\begin{tabular}{|c|r|}
					$n$ & $\rep_\Fcal(n)$ \\ \hline
					7  & 01010   \\
					6  & 01001   \\
					5  & 01000   \\
					4  & 00101   \\
					3  & 00100   \\
					2  & 010     \\
					1  & 001     \\
					0  & 0\\
					-1  & 1        \\
					-2  & 100        \\
					-3  & 10010      \\
					-4  & 10001      \\
					-5  & 10000      \\
				\end{tabular}};
\node at (9.6,1) {\scriptsize
				\begin{tabular}{|c|r|}
					$n$ & $\rep_\Fcal(n)$ \\ \hline
					20 & 0101010 \\
					19 & 0101001 \\
					18 & 0101000 \\
					17 & 0100101 \\
					16 & 0100100 \\
					15 & 0100010 \\
					14 & 0100001 \\
					13 & 0100000 \\
					12 & 0010101 \\
					11 & 0010100 \\
					10 & 0010010 \\
					9  & 0010001 \\
					8  & 0010000 \\
					\hline
					-6  & 1001010    \\
					-7  & 1001001    \\
					-8  & 1001000    \\
					-9  & 1000101    \\
					-10 & 1000100   \\
					-11 & 1000010   \\
					-12 & 1000001   \\
					-13 & 1000000   \\
				\end{tabular}};
\end{tikzpicture}
\end{center}
            \caption{Representations in the numeration system $\Fcal$ of $n
            \in [-13,21[$.}
			\label{fig:numeration-system-F}
\end{figure}

We now extend that numeration system to $\Z^2$. If $\bn=(n_1,n_2)\in\Z^2$ is
such that $|n_1|>>|n_2|$, then the word representing $n_2$ is smaller than
the word representing $n_1$. We handle this issue by padding the smaller word
so that it becomes of the same length as the longer one. The padding is done
differently according to the sign of the number involved: nonnegative numbers
are padded with $00$, while negative numbers are padded with $10$.  This is
consistent with the numeration system $\Fcal$, refer also to the two loops in
Figure~\ref{fig:numeration-system-F}.

\begin{definition}[Numeration system $\Fcal$ for $\Z^2$]
    Let $\bn=(n_1,n_2)\in\Z^2$.
    We define
    \[
        \rep_\Fcal(\bn)=\left(
        \begin{array}{c}
            \pad_t(\rep_\Fcal(n_1))\\
            \pad_t(\rep_\Fcal(n_2))
        \end{array}
        \right)
    \]
    where $t=\max\{|\rep_\Fcal(n_1)|,
                      |\rep_\Fcal(n_2)|\}$
                      and
    \[
        \pad_t(w)=
        \begin{cases}
            (00)^{\frac{1}{2}(t-|w|)}w & \text{ if } w\in0\Sigma^*,\\
            (10)^{\frac{1}{2}(t-|w|)}w & \text{ if } w\in1\Sigma^*.
        \end{cases}
    \]
\end{definition}

	The set of all canonical representations $\rep_\Fcal(\bn)$ for $\bn\in\Z^2$ are words
	in $\left\{ 
    \left(\begin{smallmatrix}
	0 \\ 0
    \end{smallmatrix}\right),
    \left(\begin{smallmatrix}
	0 \\ 1
    \end{smallmatrix}\right),
    \left(\begin{smallmatrix}
	1 \\ 0
    \end{smallmatrix}\right),
    \left(\begin{smallmatrix}
	1 \\ 1
    \end{smallmatrix}\right)
	\right\}^*$ of odd length
	such that there are no consecutive one's in each row. 
	E.g. $\rep_\Fcal(-2,9) = 
    \left(\begin{smallmatrix}
	1010\textcolor{red}{100} \\ \textcolor{red}{0010001}
    \end{smallmatrix}\right),
	\rep_\Fcal(14,2) = 
    \left(\begin{smallmatrix}
	\textcolor{red}{0100001} \\ 0000\textcolor{red}{010}
    \end{smallmatrix}\right)$.
The length of the representation splits $\Z$ and $\Z^2$ into levels.

\begin{lemma}\label{lem:level}
    For every $k\in\N$, we have
    \begin{align*}
        &\{n\in\Z\colon |\rep_\mathcal{F}(n)| = 2k+1\}
         = I_{k}\setminus I_{k-1},\\
        &\{\bn\in\Z^2\colon |\rep_\mathcal{F}(\bn)| = 2k+1\}
         = I_{k}^2\setminus I_{k-1}^2
    \end{align*}
    where $I_k = \{i\in\Z\mid -F_{2k}\leq i < F_{2k+1}\}$ for $k\geq0$ and $I_{-1}=\varnothing$.
\end{lemma}

\begin{proof}
    The first equality follows from the fact that
    $I_k=\{n\in\Z\colon |\rep_\mathcal{F}(n)| \leq 2k+1\}$
    where the minimal value $-F_{2k}$ is attained by the word $1(00)^k$ 
    and the maximal value $F_{2k+1}-1$ is attained by the word $0(10)^k$.
    The second equality follows from the fact that
    $I_k^2=\{\bn\in\Z^2\colon |\rep_\mathcal{F}(\bn)| \leq 2k+1\}$.
    \qed
\end{proof}

In Figure~\ref{fig:configuration}, the levels
    $I_{0}^2\setminus I_{-1}^2$,
    $I_{1}^2\setminus I_{0}^2$ and
    $I_{2}^2\setminus I_{1}^2$
are shown in yellow, green and blue respectively.

\section{An Automaton not only for Nonnegative Integers}
\label{sec:the-automaton-for-Z}

We introduce the terms based on \cite{zbMATH05707089}
to be used in this section.
Let $\sigma:A\mapsto A^*$ be a non-erasing morphism prolongable on the
letter $a\in A$ such that $x=(x_n)_{n\geq0}=\sigma^\omega(a)$ is infinite.
Let $C=\left\{0,...,\max_{b\in A}|\sigma(b)|-1\right\}$ be an alphabet. 
The deterministic finite automaton with output (DFAO)
associated to the morphism $\sigma$ and letter $a$ is
the 5-tuple,\footnote{In contrast to \cite{zbMATH05707089} we omit the coding as it is the identity map.}
$\mathcal{A}_{\sigma,a} = (A,C,\delta,a,A),$
where $\delta:A\times C \rightarrow A$ 
is a partial function 
such that $\delta(b,i)=c$ 
if and only if 
$c=u_i$ and $\sigma(b)=u_0\dots u_{|\sigma(b)|-1}$. 
Let $L$ be the language accepted by $ \mathcal{A}_{\sigma,a}$.
Then the triple $\Scal$ = ($L\setminus 0C^*$, $C$, $<$) is 
an abstract numeration system, 
where $(C,<)$ is the totally ordered alphabet 
with the natural order on $\N$ and the language $L\setminus 0C^*$ is radix 
ordered. 
Radix order $<_{rad}$ is defined for 
words $u,v\in L\setminus 0C^*$ as follows: $u<_{rad}v$ 
if and only if $|u| < |v|$ or $|u| = |v|$ and $u<_{lex}v$.
The map $\rep_{\mathcal{S}}:\N\rightarrow L\setminus 0C^*$
maps $n\in\N$ to the $(n+1)^{th}$ word in the 
language $L\setminus 0C^*$ and the map $\val_\Scal:L\rightarrow \N$ maps a word $w$ to the number $n$
such that $\rep_{\mathcal{S}}(n) = w'$, where $w = 0^pw'$ for a $p\geq0.$ 
For some state $r \in A$ and word $w\in C^*$, we denote by
$\Acal_{\sigma,a}(r,w)$ the state reached by the automaton after following the
path labeled by $w$ from the state $r$.
We denote it by $\Acal_{\sigma,a}(w)$ when $r$ is the initial state.
The following is essentially a reformulation of
\cite[Corollary 3.4.14]{zbMATH05707089}.

\begin{proposition}\label{decomposition}
	For every integer $n>0$,
	there exist integers $m\in\N$ and $\ell\in C$ such that
	$x_n = \sigma(x_m)[\ell]$ and $ \rep_\Scal(n) = \rep_\Scal(m)\cdot \ell$.
	Moreover, for any $i \in \N$, $|\rep_\Scal(n)| = i$ if and only if $|\sigma^{i-1}(a)| \leq n  < |\sigma^i(a)|$. 
\end{proposition}

\begin{proof}
    Let $n\in\N$. Let $u\in L \setminus 0C^*$ 
    such that $\rep_\Scal(n)=u$.
    Let $w\in C^*$ and $\ell\in C$ such that $u=w\ell$.
    We have $\val_\Scal(w\ell)=n$.
    Let $m=\val_\Scal(w)$.
    Since $w\in L \setminus 0C^*$, then $\rep_\Scal(m)=w$.
    From \cite[Corollary 3.4.14]{zbMATH05707089},
    we have
    \begin{align*}
        \sigma(x_{\val_\Scal(w)})
        =   x_{\val_\Scal(w0)}
            x_{\val_\Scal(w1)}
            \cdots
            x_{\val_\Scal(w\ell)}
            \cdots
            x_{\val_\Scal(w\cdot (K-1))}
    \end{align*}
    where
    $K=|\sigma(x_{\val_\Scal(w)})|$.
    Thus $x_n=x_{\val_\Scal(w\ell)}
    = \sigma(x_{\val_\Scal(w)})[\ell]
    =\sigma(x_m)[\ell]$.
    The other statement follows from the equation \cite[(3.12)]{zbMATH05707089}. \qed
\end{proof}

Let $\varphi$ be the morphism $\varphi:a\mapsto ab, b\mapsto a$. 
The automaton $\Acal_{\varphi^2,a}$ associated to the right-infinite fixed point of $\varphi^2$
starting with letter $a$ is shown in Figure~\ref{fig:automaton-fibo}.
We construct another automaton $\Acal_{\varphi^2,s}$
associated to the bi-infinite fixed point of $\varphi^2$
defined from the seed $s=b.a$, see Figure~\ref{fig:automaton-fibo}.
The bi-infinite Fibonacci word is
$x=\lim_{k \rightarrow +\infty}\varphi^{2k}(b.a)$, where the dot
represents the origin between positions $-1$ and $0$.
When referring to $\varphi^{2k}(b.)$ we mean the finite
word $\varphi^{2k}(b)=x_{-\mid \varphi^{2k}(b) \mid} \dots x_{-2} x_{-1}$. 

    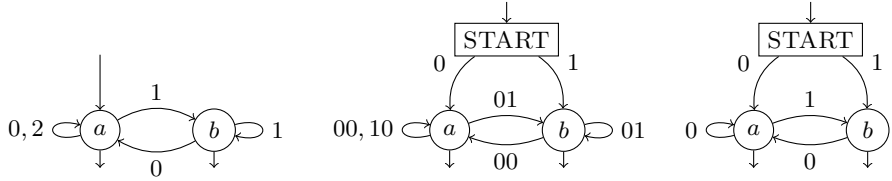
\begin{figure}[h]
        \begin{center}
			\begin{tikzpicture}[auto]
				\begin{scope}[xshift=-.1cm]
				\node[draw,circle] (A) at (0,0) {$a$};
				\node[draw,circle] (B) at (1.5,0) {$b$};
				\draw[bend left,->] (A) to node {$1$} (B);
				\draw[bend left,->] (B) to node {$0$} (A);
				\draw[loop left,->] (A) to node {$0,2$} (A);
				\draw[loop right,->] (B) to node {$1$} (B);
				\draw[<-] (A) -- ++ (0,1) node[above] {};
				\draw[->] (A) -- ++ (0,-0.5);
				\draw[->] (B) -- ++ (0,-0.5);
				\end{scope}
                \begin{scope}[xshift=4.5cm]
				\node[draw] (S) at (.75,1.2) {$\textsc{START}$};
				\node[draw,circle] (A) at (0,0) {$a$};
				\node[draw,circle] (B) at (1.5,0) {$b$};
				\draw[bend left=20,->] (A) to node {$01$} (B);
				\draw[bend left=20,->] (B) to node {$00$} (A);
                \draw[bend right,->] (S) to node[swap] {$0$} (A);
				\draw[bend left,->] (S) to node {$1$} (B);
				\draw[<-] (S) -- ++ (0,.5) node[above] {};
				\draw[loop left,->] (A) to node {$00,10$} (A);
				\draw[loop right,->] (B) to node {$01$} (B);
				\draw[->] (A) -- ++ (0,-0.5);
				\draw[->] (B) -- ++ (0,-0.5);
				\end{scope}
				\begin{scope}[xshift=8.5cm]
				\node[draw] (S) at (.75,1.2) {$\textsc{START}$};
				\node[draw,circle] (A) at (0,0) {$a$};
				\node[draw,circle] (B) at (1.5,0) {$b$};
				\draw[bend left=20,->] (A) to node {$1$} (B);
				\draw[bend left=20,->] (B) to node {$0$} (A);
                \draw[bend right,->] (S) to node[swap] {$0$} (A);
				\draw[bend left,->] (S) to node {$1$} (B);
				\draw[<-] (S) -- ++ (0,.5) node[above] {};
				\draw[loop left,->] (A) to node {$0$} (A);
				\draw[->] (A) -- ++ (0,-0.5);
				\draw[->] (B) -- ++ (0,-0.5);
				\end{scope}
			\end{tikzpicture}
        \end{center}
			\caption{Automata $\Acal_{\varphi^2,a}$,
                               $\Acal_{\varphi^2,s}$ and 
                               $\Acal_{\varphi,s}$ with seed $s=b.a$.}
			\label{fig:automaton-fibo}
	\end{figure}

\begin{lemma}\label{lem:both-signs-inductive-step-representation}
	Let $x = \lim_{k \rightarrow +\infty}\varphi^{2k}(b.a)$ and let $n\in\Z	\setminus I_0$.
	Then there exist integers $m\in\Z$ and $ 0\leq \ell < 3$ such that
	$$ x[n] = \varphi^2(x[m])[\ell], \text{ with } \rep_\mathcal{F}(n) = \rep_\mathcal{F}(m)h(\ell),$$
    where $h$ is the morphism
	$h: \left\{0,1,2\right\}^*\rightarrow\left\{0,1\right\}^*$
	defined as $h: 0\mapsto 00, 1\mapsto 01, 2\mapsto 10$.
    Moreover, if $n \in I_i\setminus I_{i-1}$ for some $1\leq i \leq k,$ then $m \in I_{i-1}\setminus I_{i-2}.$
\end{lemma}

\begin{proof}
    First assume that $n>0$.
    Let $\Acal_{\varphi^2,a}$ and $\Gcal=(L\setminus 0C^*,C,<)$. Then $x =
    \lim_{k \rightarrow +\infty}\varphi^{2k}(a)$ is $\Gcal$-automatic, i.e., for
    any $n\in\N,$ we have $x_n = \Acal_{\varphi^2,a}(\rep_{\Gcal}(n)),$
	where 
	$\rep_{\Gcal}(n) \in \{\varepsilon, 1, 2, 10, 11, 20, 21, 22, 100, 101, \dots\}$.
    From the definition of $\Acal_{\varphi^2,s}$
    in Figure~\ref{fig:automaton-fibo}, we have
    $x_n = \Acal_{\varphi^2,s}(0h(\rep_{\Gcal}(n)))$.
    Moreover, from Proposition \ref{decomposition}, 
	for $n>0$ there exist 
	integers $m\in\N$ and $ \ell$ such that $h(\rep_\Gcal(n)) = h(\rep_{\Gcal}(m))h(l).$
	On the other hand, $0h(\rep_\Gcal(n)) = \rep_\mathcal{F}(n)$
	for all $n \in \N.$ This follows from the following reasons. 
    Applying Proposition \ref{decomposition}, we obtain that
	$ |0h(\rep_\Gcal(n))| = 2i+1 \text{ if and only if } n \in 
    (I_i\setminus I_{i-1})\cap\N$. The alphabet $h(C)$ 
	has the same ordering as the alphabet $C$. Finally, a word $12$ is not accepted by the automaton $\mathcal{A}_{\varphi^2,a}$ and therefore a word $11$ is forbidden in $0h(L)$.
    Moreover, as $|h(\ell)|=2$, we observe $|\rep_{\mathcal{F}}(m)|=|\rep_{\mathcal{F}}(n)|-2$, 
    thus $m \in I_{i-1}\setminus I_{i-2}.$
	
%

    If $n = -2,$ then we denote $\ell = 0, m = -1.$ Let $n < -2$ and let
    $k\geq 2$ be such that $n \in I_k\setminus I_{k-1}.$
    We have
    \begin{equation}\label{R}
    	\varphi^{2k}(b.)[n]=\varphi^{2k-1}(a)[n+|\varphi^{2k}(b)|],
    \end{equation}
    where $|\varphi^{2k}(b)|=F_{2k}.$ As $\varphi^{2k-1}(a)$ is a prefix of $\varphi^{2k}(a),$ we can write
    $$  \varphi^{2k-1}(a)[n+F_{2k}]=\varphi^{2k}(a)[n+F_{2k}].$$
    Then $0\leq n+F_{2k} < F_{2k-1}$ and we denote $0 \leq i \leq k-1$ such that $n+F_{2k} \in I_i\setminus I_{i-1}.$ In case that $i > 0$,
    we use the previous paragraph for positive $n>0$ 
    on a long enough prefix $z$ of $x$ ($z = \varphi^{2k-2}(a)$, therefore $n+F_{2k}<|z|=F_{2k-1}$)
    and we find $m_P \in I_{i-1}\setminus I_{i-2}$ and $\ell$ such that $\rep_\Fcal(n+F_{2k})=\rep_\Fcal(m_P)h(\ell)$ and 
    $$ \varphi^2(\varphi^{2k-2}(a))[n+F_{2k}]
    = \varphi^2(\varphi^{2k-2}(a)[m_P])[\ell].$$
    As $0 \leq m_P < F_{2k-3}<F_{2k-1} = |\varphi^{2k-2}(a)|$,
    we restrict the last 
    relation just to a prefix $\varphi^{2k-3}(a)$
    and use the relation \eqref{R}
    again to get $m=m_P-F_{2k-2} < 0$
    $$ \varphi^{2k}(b.)[n] = \varphi^2(\varphi^{2k-3}(a)[m_P])[\ell] = \varphi^2(\varphi^{2k-2}(b.)[m_P-F_{2k-2}])[\ell].$$
    The representation $\rep_{\mathcal{F}}(n) = 10w$ has the first digit
    corresponding to $-F_{2k}.$ Then $(00)^{k-i}\rep_\Fcal(n+F_{2k})=00w.$
    As $m_P\in I_{i-1}\setminus I_{i-2}$, then $m_P-F_{2k-2}\in I_{k-1}\setminus I_{k-2}$
    and
    $\rep_\Fcal(m_P-F_{2k-2})=10(00)^{k-i-1}\rep_\Fcal(m_P)$. As a whole,
    $$\rep_\mathcal{F}(n) = 10(00)^{k-i-1}\rep_{\mathcal{F}}(m_P)h(\ell)=
    \rep_{\mathcal{F}}(m_P - F_{2k-2})h(\ell) = \rep_{\mathcal{F}}(m)h(\ell).$$
    If $i=0,$ then we denote $\ell = 0, m = -F_{2k-2}$ and the statement holds true. 
    \qed
\end{proof}

We show the following result.
\begin{proposition}\label{prop:dfao-for-two-sided-fibonacci}
    The DFAO $\Acal_{\varphi,s}$ associated to the seed $s=b.a$ satisfies
	$$ x_n = \Acal_{\varphi,s}(\rep_\Fcal(n)) \text{ for all } n\in\Z.$$
\end{proposition}

\begin{proof}
	Let $\Acal_{\varphi,s}$ be the automaton shown in 
	Figure~\ref{fig:automaton-fibo}, i.e., the DFAO $\Acal_{\varphi,s} = (\left\{a,b\right\}\cup\left\{\textsc{start}\right\},\left\{0,1\right\},
	\delta, \textsc{start}, \left\{a,b\right\})$ with the partial function 
	$\delta$ such that 
	\begin{itemize}
		\item $\delta(\textsc{start},\rep_\Fcal(n)) = s_n$ 
		for every $n\in I_0=\{-1,0\}$, where $s_{-1}=b$, $s_{0}=a,$
		\item $\delta(c,i)=d$ for any $c,d\in\{a,b\}$ if and only if $\varphi(c)=u$ and $u_i=d.$ 
	\end{itemize}
	Assume $n\in I_0$. If $n=0$, then we have 
	$x_0 = a = \Acal_{\varphi,s}(0) = \Acal_{\varphi,s}(\rep_\Fcal(0)) $. If $n = -1$, then $x_{-1} = b = \Acal_{\varphi,s}(1) = \Acal_{\varphi,s}(\rep_\Fcal(-1))$.
	Induction hypothesis: we assume for some $k\in \N$ 
	that $x_m = \Acal_{\varphi,s}(\rep_\Fcal(m))$
	for all $m\in I_k\setminus I_{k-1}$. Let $n\in I_{k+1}\setminus I_k$.
	Then from Lemma \ref{lem:both-signs-inductive-step-representation}
	there exist $m\in I_k\setminus I_{k-1}$ and $\ell\in\{0,1,2\}$ such that
	$x_n = \varphi^2(x_m)[\ell]$ and $\rep_\Fcal(n)=\rep_\Fcal(m)\cdot h(\ell)$,
	where $h(\ell)=\ell_0\cdot\ell_1$ for some $\ell_0,\ell_1	\in \{0,1\}$. 
	From the induction hypothesis, we have
	\begin{align*}
		x_n &= \varphi^2(x_m)[\ell] = \varphi^2(\Acal_{\varphi,s}(\rep_\Fcal(m)))[\ell] =  \varphi(\varphi(\Acal_{\varphi,s}(\rep_\Fcal(m)))[\ell_0])[\ell_1]\\
		&= \Acal_{\varphi,s}(\rep_\Fcal(m)\cdot \ell_0\cdot\ell_1) = \Acal_{\varphi,s}(\rep_\Fcal(m) h(\ell)) = \Acal_{\varphi,s}(\rep_\Fcal(n)).\qquad \qed 
	\end{align*}
\end{proof}

\section{Two-dimensional Words, Languages and Morphisms}
\label{sec:two-dim-words}

In this section, we introduce 2-dimensional words, languages and morphisms following the notations of
\cite{zbMATH05701985,MR4226493}.
Let $k\in\N$ and $\Acal = \{0, 1, \dots, k\}$ be a finite alphabet
and let $u: \left\{ 0, \dots, n_1 - 1\right\}\times\left\{ 0, \dots, n_2 - 1\right\}
\to \mathcal{A}$ be a 2-dimensional word of shape 
$\bn = (n_1, n_2) \in \mathbb{N}^2.$ 
Let $\mathcal{A}^\bn$ denote the set of all 2-dimensional words
of shape $\bn$. We refer to the words of shape (1,2), (2,1) as to the vertical,
horizontal dominoes, respectively. 
	We represent a $2$-dimensional word $u$ of shape $(n_1, n_2)\in\mathbb{N}^2$
	as a matrix	with Cartesian coordinates:
	$$ u = \begin{pmatrix}
	u_{0,n_2-1} & \dots & u_{n_1-1, n_2-1} \\
	\dots & \dots & \dots \\
	u_{0,0} & \dots & u_{n_1-1, 0}
	\end{pmatrix}.$$


	Let ${\mathcal{A}^*}^2 = \bigcup_{\textbf{n}\in\mathbb{N}^2} \mathcal{A}^\textbf{n}$ the set of all 2-dimensional words. Let $u$, $v \in {\Acal^*}^2$ be of shape 
	$(n_1, n_2)$, $(\tilde{n}_1, \tilde{n}_2)$,
	respectively. If $n_2 = \tilde{n}_2$, the concatenation in direction $\be_1$
	is defined as a 2-dimensional word $u \odot^1 v$ of shape $(n_1 + \tilde{n}_1, n_2)$ given as
	$$ u \odot^1 v = \begin{pmatrix}
		u_{0,n_2-1} & \dots & u_{n_1-1, n_2-1} & v_{0,n_2-1} & \dots\dots & v_{\tilde{n_1}-1, n_2-1}\\
		\dots & \dots & \dots & \dots & \dots & \dots\\
		u_{0,0} & \dots & u_{n_1-1, 0} & v_{0,0} & \dots\dots & v_{\tilde{n}_1-1, 0}
	\end{pmatrix}.$$
	If $n_1 = \tilde{n}_1$, the concatenation in 
	direction $\be_2$ is defined analogically.
	A word $v \in {\Acal^*}^2$ is a \emph{subword} of a word $u\in {\Acal^*}^2$ if
	there exist words $u_1, u_2, u_3, u_4\in {\Acal^*}^2$ such that
	$u = u_3 \odot^2 (u_1 \odot^1 v \odot^1 u_2) \odot^2 u_4$.
%

A subset $L \subseteq {\mathcal{A}^*}^2$
is called a 2-dimensional \emph{factorial} language 
if $u\in L$ implies that $v\in L$ for all $2$-dimensional subwords $v$ of $u$.

Let $\Acal$ and $\Bcal$ be two alphabets.
Let $L\subseteq\Acal^{*^2}$ be a factorial language.
A function $\omega:L\to\Bcal^{*^2}$ is a \emph{$2$-dimensional
morphism} if for every
$i$ with $1\leq i\leq 2$,
and every $u,v\in L$ such that 
$u\odot^i v$ is defined
and
$u\odot^i v\in L$,
we have
that the concatenation $\omega(u)\odot^i \omega(v)$
in direction $\be_i$ is defined and
\begin{equation*}
    \omega(u\odot^i v) = \omega(u)\odot^i \omega(v).
\end{equation*}

A $2$-dimensional morphism $L\to\Bcal^{*^2}$ is thus completely defined from the
image of the letters in $\Acal$ and can be denoted as a rule $\Acal\to\Bcal^{*^2}.$

A subset $X\subseteq\Acal^{\Z^2}$ is called a \emph{subshift}
if it is closed under the shift\footnote{Note that from now on, $\sigma$ denotes the shift action and not a morphism.} $\sigma$ and closed with respect to
compact product topology. Let $L\subseteq\Acal^{*^2}$ be a factorial language
and $\Lcal(x)$ be the factorial language containing
all subwords of the configuration $x\in\Acal^{\Z^2}$.
Then, $\Xcal_L=\{x\in\Acal^{\Z^2}\mid \Lcal(x)\subset L\}$ is a subshift 
generated by $L$.
A $2$-dimensional morphism 
$\omega:L \to\Bcal^{*^2}$ 
can be extended to a continuous map 
$\omega:\Xcal_L\to\Bcal^{\Z^2}$
in such a way that the origin of $\omega(x)$ is at zero position
in the word $\omega(x_{(0,0)})$
for all $x\in\Xcal_L$.

In general, the closure under the shift of the image of a subshift
$X\subseteq\Acal^{\Z^2}$ under $\omega$
is the subshift $\overline{\omega(X)}^{\sigma}
= \{\sigma^\bk\omega(x)\in\Bcal^{\Z^2} \mid \bk\in\Z^2, x\in X\}
\subseteq \Bcal^{\Z^2}$.

A 2-dimensional morphism $\omega:\Acal\to\Acal^{*^2}$ is said \emph{expansive} if
the width and height of $\omega^k(a)$ goes to $\infty$ for all letters $a\in\Acal$.
A subshift $X\subset\Acal^{\Z^2}$ is \emph{self-similar} if there exists
an expansive 2-dimensional morphism $\Acal\to\Acal^{*^2}$ such that 
$X=\overline{\omega(X)}^{\sigma}$.

\section{Self-similarity of the Wang Shift $\Omega_\Zcal$}
\label{sec:self-similarity}


\input{SAGEOUTPUTS/macro_substitutions.tex}

\begin{proposition}\label{prop:self-similarity-OmegaZ}
    The Wang shift $\Omega_\Zcal$ is self-similar satisfying
        $\overline{\phi(\Omega_\mathcal{Z})}^\sigma = \Omega_\mathcal{Z}$
        where $\phi$ is the 2-dimensional morphism over the alphabet
        $\Hcal=\{0,\dots,15\}$
    \begin{equation}\label{eq:phi-2d}
    \begin{array}{ll}
        \phi:&\Hcal\to\Hcal^{*^2}\\[2mm]
        &\left\{\scriptsize\arraycolsep=1.8pt\subgammaOIII\right.
    \end{array}
    \end{equation}
\end{proposition}

The proof of Proposition~\ref{prop:self-similarity-OmegaZ} is done in the
appendix. It is using algorithms to desubstitute Wang shifts based on the
notion of marker tiles \cite{MR4226493}.

Similarly to the 1-dimensional case, we can build an automaton associated to 
a fixed point of the
2-dimensional morphism $\phi$ defined in Equation~\eqref{eq:phi-2d}. 
Let 
$s=\left(\begin{smallmatrix}8 & 12\\1 & 6\end{smallmatrix}\right)\in\Hcal^{(2,2)}$
be the seed associating one letter to each quadrant.
We observe that $\phi^2(s)$ prolongates $s$ at the origin.
Therefore, $\lim_{k\to\infty}\phi^{2k}(s)$ defines a configuration in
$\Hcal^{\Z^2}$ which is a fixed point of $\phi^2$.

Associated to the morphism $\phi$ and to the seed
$s=\left(\begin{smallmatrix}s_{(-1,0)} & s_{(0,0)}\\s_{(-1,-1)} & s_{(0,-1)}\end{smallmatrix}\right)\in\Hcal^{(2,2)}$,
we construct a DFAO $\mathcal{A}_{\phi,s} =
(\Hcal\cup\{\textsc{start}\},\Sigma,\delta,I,\Hcal)$
such that
$\Sigma = \left\{ \left(\begin{smallmatrix} 0 \\ 0 \end{smallmatrix}\right),
		          \left(\begin{smallmatrix} 0 \\ 1 \end{smallmatrix}\right),
		          \left(\begin{smallmatrix} 1 \\ 0 \end{smallmatrix}\right),
		          \left(\begin{smallmatrix} 1 \\ 1 \end{smallmatrix}\right) \right\}$,
$I = \left\{\textsc{start}\right\}$ and
$\delta: Q\times \Sigma \rightarrow Q$ is a partial function such that
\begin{itemize}
    \item $\delta(\textsc{start},\rep_\Fcal(\bn)) = s_\bn$ 
        for every $\bn\in I_0^2=\{(0,0),(-1,0),(0,-1),(-1,-1)\}$,
    \item $\delta(a,e)=b$ for any $a,b\in\mathcal{H}$ and $e\in\Sigma$ if and only if 
    $b$ is in $\phi(a)$ at position $e$.
\end{itemize} 

The automaton $\Acal_{\phi,s}$ associated to the morphism $\phi$ 
and seed
$s=\left(\begin{smallmatrix}8 & 12\\1 & 6\end{smallmatrix}\right)$
is shown in Figure~\ref{fig:automaton}.

\begin{figure}[h]
    \begin{center}
        \includegraphics[width=\linewidth]{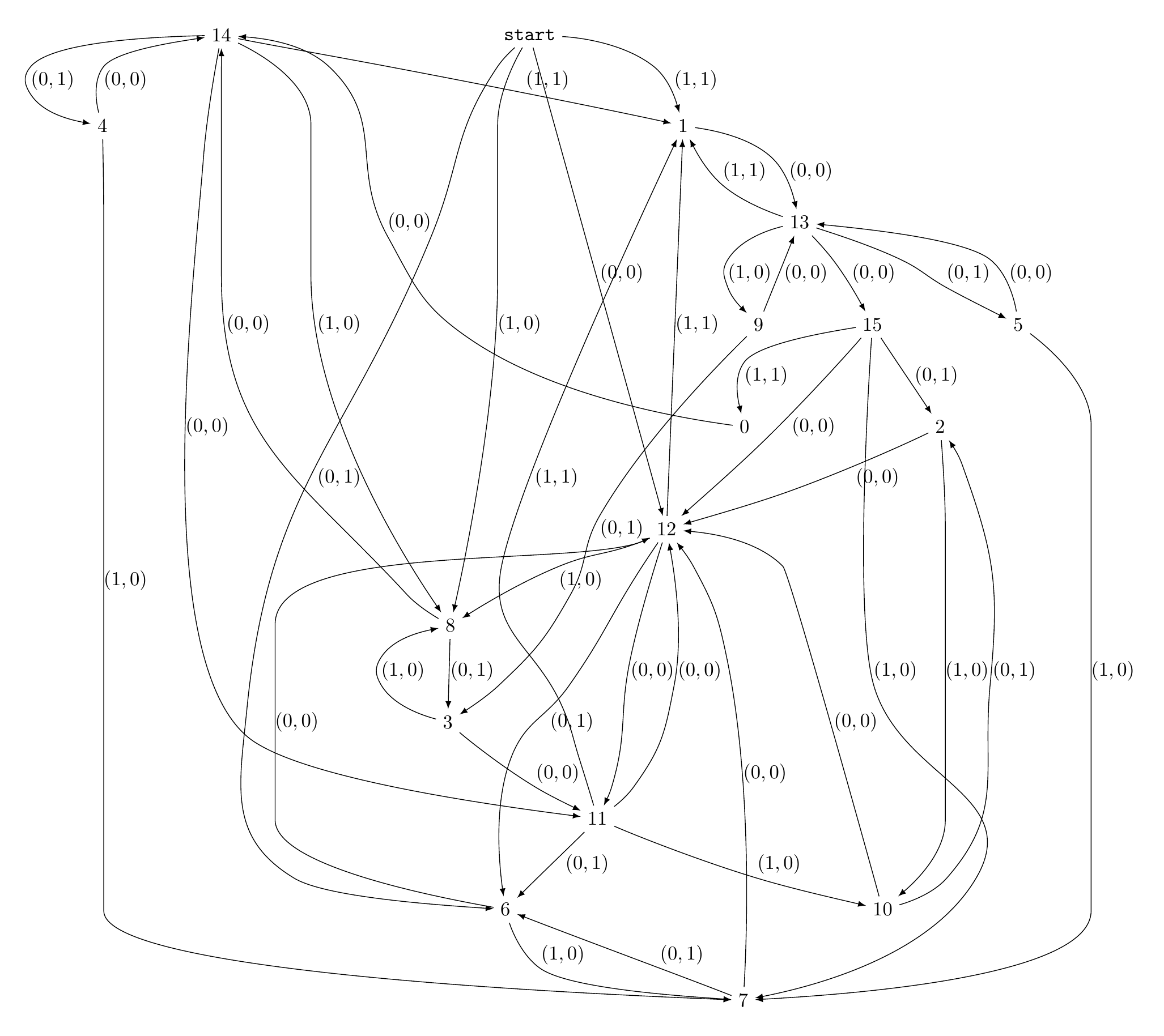}
    \end{center}
    \caption[The automaton $\Acal_{\phi,s}$ associated to the 2-dimensional
             morphism $\phi$ and seed $s$.]
            {The automaton $\Acal_{\phi,s}$ associated to the 2-dimensional
             morphism $\phi$ and seed $s=\left(\begin{smallmatrix}8 & 12\\1 &
             6\end{smallmatrix}\right)$.}
    \label{fig:automaton}
\end{figure}

\section{Proof of Main Results}
\label{sec:proof-main-results}

In this section, we prove Theorem~\ref{theoremA}.
The strategy is to extract the horizontal and vertical structure (expressed as
1-dimensional morphisms) of a 2-dimensional morphism and exploit results
proved for the 1-dimensional case in Section~\ref{sec:the-automaton-for-Z},
in particular Lemma~\ref{lem:both-signs-inductive-step-representation}.

Let $\omega$ be a 2-dimensional morphism on the alphabet $\Qcal$
and $\Xcal_\omega$ be the associated substitutive subshift.
Since $\omega:\Xcal_\omega\to\Xcal_\omega$ is well-defined, it imposes that
the horizontal width of 
$\omega(a)$
equals the horizontal width of 
$\omega(b)$ for every pair of letters $a,b\in\Qcal$ appearing in the same column.
This holds also for the height of the images of letters appearing in the same row.
However, more can be said.

We define $\sim_{\rm col}$ the equivalence relation as the reflexive,
symmetric and transitive closure of the relation
$ \{(a, b) \mid \left(\begin{smallmatrix} b \\ a	\end{smallmatrix}\right)
    \in \mathcal{L}_\omega\}$
made of the vertical dominoes in the language.
We define $\sim_{\rm row}$ the equivalence relation as the reflexive,
symmetric and transitive closure of the relation
$ \{(a, b) \mid \left(\begin{smallmatrix} a & b	\end{smallmatrix}\right)
    \in \mathcal{L}_\omega\}$
made of the horizontal dominoes in the language.
We have that $a\sim_{\rm col}b$ if and only if letters $a$ and $b$ appear in the same column in
some configuration of $\Xcal_\omega$.
Let $\pi_{\rm col}:\Qcal \rightarrow \Qcal|_{\sim_{\rm col}}$ 
and $\pi_{\rm row}:\Qcal \rightarrow \Qcal|_{\sim_{\rm row}}$ 
be the maps defined as
$ \pi_{\rm col}: a \mapsto [a]_{\sim_{\rm col}}$,
 and  $\pi_{\rm row}: a \mapsto [a]_{\sim_{\rm row}}$
mapping a letter $a \in \Qcal$ to its equivalence class.

We define the horizontal and vertical structure of $\omega$ resp. as the
1-dimensional morphisms
$\omega_\textsc{horiz}:\Qcal|_{\sim_{\rm col}}^*\to\Qcal|_{\sim_{\rm col}}^*$
and
$\omega_\textsc{vert}:\Qcal|_{\sim_{\rm row}}^*\to\Qcal|_{\sim_{\rm row}}^*$ by
    \[
    \omega_\textsc{horiz}([a]_{\sim_{\rm col}})
        =[w_{0,0}]_{\sim_{\rm col}} \cdots [w_{m-1,0}]_{\sim_{\rm col}},
    \]
    \[
        \omega_\textsc{vert}([a]_{\sim_{\rm row}})
        =[w_{0,0}]_{\sim_{\rm row}} \cdots [w_{0,n-1}]_{\sim_{\rm row}}
    \]
where $w=\omega(a)\in\Qcal^{(m,n)}$.

\begin{lemma}\label{lemma1}
	Let $\phi$ be the 2-dimensional morphism from Proposition \ref{prop:self-similarity-OmegaZ} and $x$ the point fixed by $\phi^2,$
	$x = \lim_{k \rightarrow +\infty}\phi^{2k}(s)$ for the seed $s=\left(\begin{smallmatrix}8&12\\1&6\end{smallmatrix}\right)$.
	Let $\bn\in\Z^2\setminus I_0^2$. Then, there exist vectors $\bm\in\Z^2$ and $
	\bell\in\left\{0,1,2\right\}^2$ such that
	$$x_{\bn}=\phi^2(x_{\bm})[\bell] \text{, where } \rep_\Fcal(\bn) = \rep_\Fcal(\bm)\cdot h(\bell).$$
\end{lemma}

\begin{proof}
    Let $\bn\in\Z^2\setminus I_0^2$.
    The powers $\phi^{2k}(s)$ are defined such that they grow in all four quadrants.
    The vertical and horizontal structure of $\phi$ are respectively
    $ \phi_\textsc{vert} = \phi_\textsc{horiz} = 
        \left\{\begin{array}{l}
                a\mapsto ab\\
                b\mapsto a
               \end{array}\right.$
    satisfying
    \begin{align*}
        \width(\phi(w)) &= |\phi_\textsc{horiz}\circ\pi_{col}(w_{0,0}\cdots w_{m-1,0})|\\
        \height(\phi(w)) &= |\phi_\textsc{vert}\circ\pi_{row}(w_{0,0}\cdots w_{0,n-1})|
    \end{align*}
    for all 2-dimensional words $w\in\Lcal_\phi$ of shape $(m,n)$.
    Therefore, the vectors $\bm$ and $\bell$ we are searching for can be found
    coordinate by coordinate.

    Let $y=\lim_{k\to\infty}\phi_\textsc{horiz}^{2k}(b.a)
      =\lim_{k\to\infty}\phi_\textsc{vert}^{2k}(b.a)$.
    From Lemma~\ref{lem:both-signs-inductive-step-representation},
	there exist integers $m_1,m_2\in\Z$ and $ 0\leq \ell_1,\ell_2 < 3$ such that
    \[
        y[n_1] = \phi_\textsc{horiz}^2(y[m_1])[\ell_1]
        \qquad
        \text{ and }
        \qquad
        y[n_2] = \phi_\textsc{vert}^2(y[m_2])[\ell_2],
    \]
	where 
        $\rep_\Fcal(n_1) = \rep_\Fcal(m_1)\cdot h(\ell_1)$ and 
        $\rep_\Fcal(n_2) = \rep_\Fcal(m_2)\cdot h(\ell_2)$.
    Moreover, it satisfies
    $x_{(n_1,n_2)}=\phi^2(x_{(m_1,m_2)})[(\ell_1,\ell_2)]$.
    We conclude
    \[
        \rep_\Gcal(\bn) = 
        \left(\begin{array}{l}
            \pad_t(\rep_\Fcal(n_1)) \\
            \pad_t(\rep_\Fcal(n_2)) 
        \end{array}\right)
        = 
        \left(\begin{array}{l}
            \pad_{t-2}(\rep_\Fcal(m_1)\cdot h(\ell_1)\\
            \pad_{t-2}(\rep_\Fcal(m_2)\cdot h(\ell_2)
        \end{array}\right)
	    = \rep_\Fcal(\bm)\cdot h(\bell).
    \]
    where $t=\max\{|\rep_\Fcal(n_1)|,|\rep_\Fcal(n_2)|\}$.\qed
\end{proof}


\begin{lemma}\label{lemma3}
	For any state $r\in Q\setminus\{\textsc{start}\}$ in the automaton $\mathcal{A}_{\phi,s}$ and any 
    $\bell\in\left\{0,...,\width(\phi^2(r))-1\right\}\times 
             \left\{0,...,\height(\phi^2(r))-1\right\}$ we have  
	$$\mathcal{A}_{\phi,s}(r,h(\bell)) = \phi^2(r)[\bell].$$
\end{lemma}

\begin{proof}
	Let $r$ and $\bell$ be according to the assumptions. Therefore,
	$\bell \in \left\{0,1,2\right\}^2.$
	Then, there exist unique vectors $\bell_0,\bell_1\in
	\left\{0,1\right\}^2$ such that $\phi^2(s)[\bell] = \phi(\phi(s)[\bell_0])[\bell_1].$
	As $\bell_0, \bell_1 \in \left\{0,1\right\}^2$, they belong to the set of edges $\Sigma$ of the automaton $\mathcal{A}_{\phi,s}$.
	Then, using the general properties of an automaton we have
	\begin{align*}
		\phi^2(r)[\bell] 
            &= \phi(\phi(r)[\bell_0])[\bell_1]
			 = \phi(\mathcal{A}_{\phi,s}(r, \bell_0))[\bell_1]\\
			&= \mathcal{A}_{\phi,s}(\mathcal{A}_{\phi,s}(r, \bell_0),\bell_1)
			 = \mathcal{A}_{\phi,s}(r, \bell_0\bell_1)
			 = \mathcal{A}_{\phi,s}(r, h(\bell)),
	\end{align*}
	where the last equation holds for the following reasons.
	
	I) If $\bell\in\left\{0,1\right\}^2, \left(\begin{smallmatrix}	i \\ j	\end{smallmatrix}\right) = \bell,$
	then $h(\bell) =\left(\begin{smallmatrix}	0 & i \\ 0 & j	\end{smallmatrix}\right).$
	Also, 
	$\bell_0 = \left(\begin{smallmatrix}	0 \\ 0	\end{smallmatrix}\right)$ and
	$\bell_1 = \left(\begin{smallmatrix}	i \\ j	\end{smallmatrix}\right).$
	
	II) If $\ell \in \left\{2\right\}\times\left\{0,1\right\}, 
	\left(\begin{smallmatrix}	2 \\ j	\end{smallmatrix}\right) = \bell,$
	then $h(\bell) =\left(\begin{smallmatrix}	1 & 0 \\ 0 & j	\end{smallmatrix}\right).$
	On the other hand,
	$\bell_0 = \left(\begin{smallmatrix}	1 \\ 0	\end{smallmatrix}\right)$ and
	$\bell_1 = \left(\begin{smallmatrix}	0 \\ j	\end{smallmatrix}\right).$
	The case $\bell \in \left\{0,1\right\}\times\left\{2\right\}$ is analogical.
	
	III) If $\bell \in \left\{2\right\}^2,$
	then $h(\bell) =\left(\begin{smallmatrix}	1 & 0 \\ 1 & 0	\end{smallmatrix}\right).$
	Also,
	$\bell_0 = \left(\begin{smallmatrix}	1 \\ 1	\end{smallmatrix}\right)$ and
	$\bell_1 = \left(\begin{smallmatrix}	0 \\ 0	\end{smallmatrix}\right).$
    \qed
\end{proof}

\begin{theorem}\label{thm:our-Cobham}
	Let $\phi$ be a 2-dimensional morphism and $x$ the point fixed by $\phi^2$,
	$x = \lim_{k \rightarrow +\infty}\phi^{2k}(s)$ for a seed $s=\left(\begin{smallmatrix}8 & 12\\1 & 6\end{smallmatrix}\right)$.
	Then, there exists an automaton $\mathcal{A}$ such that
	$ x_{\bn} = \mathcal{A}(\rep_{\mathcal{F}}(\bn))$.
\end{theorem}

\begin{proof}
	Let $\Acal = \Acal_{\phi,s}$ the automaton associated to the morphism $\phi$ 
	and seed
	$s=\left(\begin{smallmatrix}8 & 12\\1 & 6\end{smallmatrix}\right)$.
	If $\bn \in I_0^2=\left\{(0,0), (-1,0), (-1,-1), (0,-1)\right\},$
	then $x_{\bn} = \mathcal{A}(\rep_{\mathcal{F}}(\bn))$.
	Induction hypothesis: 
    we assume for some $k\in\N$ that $x_{\bm} = \mathcal{A}(\rep_{\mathcal{F}}(\bm))$
    for all $\bm \in I_{k}^2\setminus I_{k-1}^2$.
	Let $\bn\in I_{k+1}^2\setminus I_k^2.$ Then, from Lemma \ref{lemma1} there exist
    $\bm\in\Z^2$ and $\bell\in\{0,1,2\}^2$ such that
	$x_{\bn}=\phi^2(x_{\bm})[\bell]$
	where
	$\rep_\Fcal(\bn) = \rep_\Fcal(\bm)\cdot h(\bell)$.

	This implies $|\rep_\Fcal(\bm)| = |\rep_\Fcal(\bn)| - 2,$
    and therefore by Lemma~\ref{lem:level}, $\bm\in I_{k}^2\setminus I_{k-1}^2.$
	From the induction hypothesis, 
	$x_{\bm} = \mathcal{A}(\rep_\Fcal(\bm))$.
	Then, from the induction hypothesis and Lemma \ref{lemma3}, we have
	\begin{align*}
		x_{\bn}
            &=\phi^2(x_{\bm})[\bell]
			 =\phi^2(\mathcal{A}(\rep_\Fcal(\bm)))[\bell]\\
            &=\mathcal{A}(\rep_\Fcal(\bm)h(\bell))
			 =\mathcal{A}(\rep_\Fcal(\bn)). \qquad\qed
	\end{align*}
\end{proof}

\begin{proof}[of Theorem~\ref{theoremA}]
	Let $\phi$ be the 2-dimensional morphism from 
    Proposition~\ref{prop:self-similarity-OmegaZ} and let
	$x=\phi^2(x)$ be the point fixed by $\phi^2,$
	where $x = \lim_{k \rightarrow +\infty} \phi^{2k}(s)$ of the seed $s=\left(\begin{smallmatrix}8 & 12\\1 & 6\end{smallmatrix}\right)$.
	Let $\Acal = \mathcal{A}_{\phi,s}$ (see Figure \ref{fig:automaton}).
	The conclusion follows from Theorem~\ref{thm:our-Cobham}.
    \qed
\end{proof}

\begin{example}
	Let $\bn=(-1,6)\in \Z^2$. Then, $\rep_\Fcal(\bn)=\left(
	\begin{smallmatrix}10101\\01001\end{smallmatrix}\right)$
	and $\Acal_{\phi,s}$ gives
	$$ \textsc{start}\xrightarrow{(1,0)}8\xrightarrow{(0,1)}3\xrightarrow{(1,0)}8
	\xrightarrow{(0,0)}14\xrightarrow{(1,1)}1.$$
	The tile at position $\bn$ in the tiling $x$ in Figure \ref{fig:configuration}
	is indeed $x_\bn = 1$. 
\end{example}

Since $\Omega_\Zcal$ is minimal, we believe that 
Theorem~\ref{theoremA} can be extended to all configurations
in $\Omega_\Zcal$ provided an additional input is given. 
Moreover, we believe that Theorem~\ref{theoremA} holds for a large family of
self-similar subshifts and not only for Fibonacci-like examples,
thus extending Cobham's theorem to $\Z^2$ and $\Z^d$.
This asks for further research and is part of an ongoing work.

\bibliographystyle{alpha} 
\bibliography{biblio}

\newpage
\section{Appendix}
\label{sec:appendix}

We now define the notion of markers for subshifts $X\subset\Acal^{\Z^d}$.
Their existence allows to desubstitute uniquely the configurations in $X$ 
using a $d$-dimensional morphism.
Originally, those results were proved for $d=2$ in order
to desubstitute configurations from Wang shifts, see
\cite{MR3978536} and \cite{MR4226493}.
It turns out that the notion of markers is more general and the results holds
in general subshifts $X\subset\Acal^{\Z^d}$ \cite{labbe_three_2020}.

Recall that if $w:\Z^d\to\Acal$ is a configuration and $a\in\Acal$ is a letter, then
$w^{-1}(a)\subset\Z^d$ is the set of positions where the letter $a$ appears in
$w$.

\begin{definition}\label{def:markers}
    Let $\Acal$ be an alphabet
    and $X\subset\Acal^{\Z^d}$ be a subshift.
    A nonempty subset $M\subset\Acal$ is called \emph{markers in the
    direction $\be_i$}, with $i\in\{1,\dots,d\}$,
    if positions of the letters of $M$ in any configuration are nonadjacent $(d-1)$-dimensional
    layers orthogonal to $\be_i$, that is,
    for all configurations $w\in X$
    there exists $P\subset\Z$ such that
    the positions of the markers satisfy
    \begin{equation*}
        w^{-1}(M) = P\be_i +\sum_{k\neq i} \Z\be_k
        \quad
        \text{ with }
        \quad
        1\notin P-P
    \end{equation*}
    where $P-P=\{b-a\mid a\in P, b\in P\}$ is the set of differences
    between elements of $P$.
\end{definition}

Note that it follows from the definition that a subset of markers is a proper
subset of $\Acal$ as the case $M=\Acal$ is impossible.

The presence of markers allows to
desubstitute uniquely the configurations of a subshift.
There is even a choice to be made in the construction of the substitution. 
We may construct the substitution in such a way that the markers are on the left or
on the right in the image of letters that are dominoes 
in the direction $\be_k$. We make this distinction in the statement of
the following result which was stated in the context of Wang shifts
in \cite{MR3978536,MR4226493} and extended to $\Z^d$ in \cite{labbe_three_2020}.

\begin{theorem}\label{thm:if-markers-desubstitute}
    {\rm\cite{labbe_three_2020}}
    Let $\Acal$ be an alphabet
    and $X\subset\Acal^{\Z^d}$ be a subshift.
    If there exists a subset
    $M\subset\Acal$ 
    of markers in the direction 
    $\be_i\in\{\be_1,\dots,\be_d\}$,
    then 
\begin{enumerate}
    \item[(i)] (markers on the right) there exists an alphabet $\Bcal_R$,
    a subshift $Y\subset\Bcal_R^{\Z^d}$
    and a $2$-dimensional morphism
    $\omega_R:Y\to X$
    such that 
    \begin{equation*}
        \omega_R(\Bcal_R)\subseteq (\Acal\setminus M)\cup 
        \left((\Acal\setminus M)\odot^i M\right)
    \end{equation*}
    which is recognizable and onto up to a shift and
\item[(ii)] (markers on the left) there exists an alphabet $\Bcal_L$,
    a subshift $Y\subset\Bcal_L^{\Z^d}$
    and a $2$-dimensional morphism
    $\omega_L:Y\to X$
    such that 
    \begin{equation*}
        \omega_L(\Bcal_L)\subseteq (\Acal\setminus M)\cup 
        \left(M\odot^i (\Acal\setminus M)\right)
    \end{equation*}
    which is recognizable and onto up to a shift.
\end{enumerate}
\end{theorem}

The morphisms provided in Theorem~\ref{thm:if-markers-desubstitute} 
can be computed with an algorithm taking as input the subset of marker tiles.
Those algorithms are available in
\cite{MR4226493,labbe_three_2020}
and are
implemented in the \texttt{slabbe} optional package 
\cite{labbe_slabbe_0_6_2_2020} of SageMath. 
The algorithms are used in the proof below.

In the following proof, we use the notation $\llbracket m,n\rrbracket=\{k\in\Z\mid m\leq
k\leq n\}$, where $m,n\in\Z$, to denote interval of integers.

\begin{proof}[of Proposition~\ref{prop:self-similarity-OmegaZ}]
    We use the algorithms to desubstitute Wang shifts
    based on the notion of marker tiles \cite[Theorem 3.6]{MR4226493}.
We construct the set $\Zcal$ of 16 Wang tiles:
{\scriptsize
\begin{verbatim}
sage: from slabbe import WangTileSet
sage: tiles = ["DOJO", "DOHL", "JMDP", "DMDK", "HPJP", "HPHN", "HKDP", "BOIO",
....:          "ILEO", "ILCL", "ALIO", "EPIP", "IPIK", "IKBM", "IKAK", "CNIP"]
sage: Z = WangTileSet([tuple(tile) for tile in tiles]); Z
Wang tile set of cardinality 16
\end{verbatim}}

We compute a subset of markers $M\subset\Zcal$ for the direction $\be_2$ and we
    use them to desubstitute the Wang shift $\Omega_\Zcal$:
{\scriptsize
\begin{verbatim}
sage: Z.find_markers(i=2, radius=2, solver="dancing_links")
[[0, 1, 2, 3, 4, 5, 6]]
sage: M = [0, 1, 2, 3, 4, 5, 6]
sage: Z1,gamma0 = Z.find_substitution(M, i=2, radius=2,
....:            side="right", solver="dancing_links")
\end{verbatim}}
\noindent
We obtain the morphism $\gamma_0:\Omega_{\Zcal_1}\to\Omega_\Zcal$ given as a rule of the form
\[
\begin{array}{ll}
    \gamma_0:&\Zrange{17}\to\Zrange{15}^{*^2}\\[2mm]
    &\left\{\scriptsize\arraycolsep=1.8pt\subgammaO\right.
\end{array}
\]
and the set $\Zcal_1$ of 18 Wang tiles
\[
    \Zcal_1 = \left\{\raisebox{-6.5mm}{\includegraphics[width=.55\linewidth]
                                      {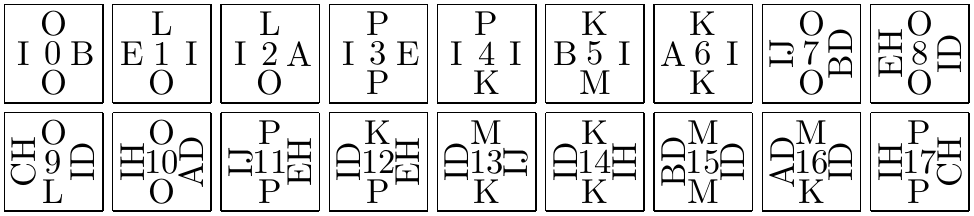}}\right\}.
\]
We compute a subset of markers $M_1\subset\Zcal_1$ for the direction $\be_1$ and we
    use them to desubstitute the Wang shift $\Omega_{\Zcal_1}$:
{\scriptsize
\begin{verbatim}
sage: Z1.find_markers(i=1, radius=1,solver="dancing_links")
[[0, 1, 2, 7, 8, 9, 10]]
sage: M = [0,1,2,7,8,9,10]
sage: Z2,gamma1 = Z1.find_substitution(M, i=1, radius=1, 
....:             side="right", solver="dancing_links")
\end{verbatim}}
\noindent
    We obtain the morphism $\gamma_1:\Omega_{\Zcal_2}\to\Omega_{\Zcal_1}$ given
    as a rule of the form
\[
\begin{array}{ll}
    \gamma_1:&\Zrange{15}\to\Zrange{17}^{*^2}\\[2mm]
    &\left\{\scriptsize\arraycolsep=1.8pt\subgammaI\right.
\end{array}
\]
and the set $\Zcal_2$ of 16 Wang tiles
\[
    \Zcal_2 = \left\{\raisebox{-6.5mm}{\includegraphics[width=.55\linewidth]
                                      {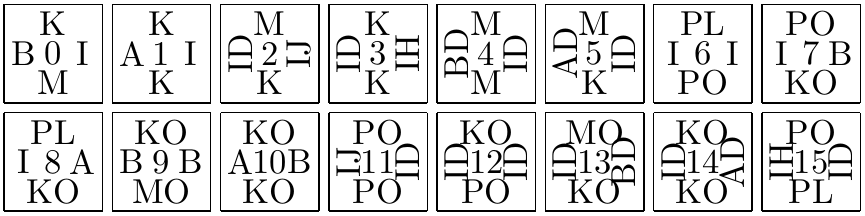}}\right\}.
\]
It turns out that $\Zcal$ and $\Zcal_2$ are equivalent. The proof can be done
using graph isomorphisms:
{\scriptsize
\begin{verbatim}
sage: Z2.is_equivalent(Z)
True
\end{verbatim}}
    \noindent
The bijection \texttt{vert} between the vertical colors,
the bijection \texttt{horiz} between the horizontal colors
and bijection $\gamma_2$ from $\Zcal_2$ to $\Zcal$ is computed as follows:
{\scriptsize
\begin{verbatim}
sage: _,vert,horiz,gamma2 = Z.is_equivalent(Z2, certificate=True)
\end{verbatim}
}
\noindent
We obtain the morphism $\gamma_2:\Omega_{\Zcal}\to\Omega_{\Zcal_2}$ given as a rule of the form
\[
\begin{array}{ll}
    \gamma_2:&\Zrange{15}\to\Zrange{15}^{*^2}\\[2mm]
    &\left\{\scriptsize\arraycolsep=1.8pt\subgammaII\right.
\end{array}
\]
and we have
\[
\begin{array}{ll}
    \gamma_0\gamma_1\gamma_2:&\Zrange{15}\to\Zrange{15}^{*^2}\\[2mm]
    &\left\{\scriptsize\arraycolsep=1.8pt\subgammaOIII\right.
\end{array}
\]
We conclude that
$\Omega_\Zcal
    =\overline{\gamma_0(\Omega_{\Zcal_1})}^\sigma
    =\overline{\gamma_0\gamma_1(\Omega_{\Zcal_2})}^\sigma
    =\overline{\gamma_0\gamma_1\gamma_2(\Omega_\Zcal)}^\sigma
    =\overline{\phi(\Omega_\Zcal)}^\sigma$.
    \qed
\end{proof}

\newpage

\begin{theorem}\label{theoremB}
	The Wang shift $\Omega_\Zcal$ generated by the tiles $\Zcal$
	is self-similar, minimal and aperiodic.
\end{theorem}

\begin{proof}[of Theorem~\ref{theoremB}]
	From Proposition~\ref{prop:self-similarity-OmegaZ}
    we obtain that $\Omega_\Zcal$ is self-similar
    satisfying $\overline{\phi(\Omega_\mathcal{Z})}^\sigma = \Omega_\mathcal{Z}$.
    Since $\phi$ is expansive and recognizable, the Wang shift $\Omega_\Zcal$ is aperiodic
    \cite[Proposition 6]{MR3978536}.

    We compute the set of $2\times 2$ words for which the morphism $\phi$ is
    prolongable for some power of $\phi$ in each of the four quadrants:
{\scriptsize
\begin{verbatim}
sage: gamma012 = gamma0 * gamma1 * gamma2
sage: seeds = flatten(gamma012.prolongable_seeds_list())
sage: seeds
[
[ 8 12]  [14 11]  [ 8 12]  [14 11]  [ 9 11]  [13 12]  [ 9 12]  [13 11]
[ 1  6], [13 12], [ 7 13], [ 6  5], [ 8 12], [ 3  6], [10 14], [ 2  4]
]
\end{verbatim}}
\noindent
We check whether these $2\times2$ factors belong to the language $\Lcal_\phi$
generated from the application of the morphism on the letters:
{\scriptsize
\begin{verbatim}
sage: seeds = [[list(reversed(col)) for col in seed.columns()] for seed in seeds]
sage: F = gamma012.list_2x2_factors()
sage: all(seed in F for seed in seeds)
True
\end{verbatim}}
\noindent
    Thus the substitutive subshift $\Xcal_\phi$ is the unique self-similar subshift
    $X\subset\Zrange{15}^{\Z^2}$ satisfying $\overline{\phi(X)}^\sigma=X$
    \cite[Lemma 3.9]{labbe_three_2020}.
    Thus $\Xcal_\phi=\Omega_\Zcal$. 
    Moreover $\Xcal_\phi$ is minimal since $\phi$ is primitive and expansive.
    \qed
\end{proof}

%
%

\end{document}

%% file: SAGEOUTPUTS/macro_substitutions.tex
\newcommand\subgammaO{\begin{array}{llll}
0\mapsto \left(7\right)
,&
1\mapsto \left(8\right)
,&
2\mapsto \left(10\right)
,&
3\mapsto \left(11\right)
,\\
4\mapsto \left(12\right)
,&
5\mapsto \left(13\right)
,&
6\mapsto \left(14\right)
,&
7\mapsto \left(\begin{array}{r}
0 \\
7
\end{array}\right)
,\\
8\mapsto \left(\begin{array}{r}
1 \\
8
\end{array}\right)
,&
9\mapsto \left(\begin{array}{r}
1 \\
9
\end{array}\right)
,&
10\mapsto \left(\begin{array}{r}
1 \\
10
\end{array}\right)
,&
11\mapsto \left(\begin{array}{r}
4 \\
11
\end{array}\right)
,\\
12\mapsto \left(\begin{array}{r}
6 \\
11
\end{array}\right)
,&
13\mapsto \left(\begin{array}{r}
2 \\
12
\end{array}\right)
,&
14\mapsto \left(\begin{array}{r}
6 \\
12
\end{array}\right)
,&
15\mapsto \left(\begin{array}{r}
3 \\
13
\end{array}\right)
,\\
16\mapsto \left(\begin{array}{r}
3 \\
14
\end{array}\right)
,&
17\mapsto \left(\begin{array}{r}
5 \\
15
\end{array}\right)
.
\end{array}}
\newcommand\subgammaI{\begin{array}{llll}
0\mapsto \left(5\right)
,&
1\mapsto \left(6\right)
,&
2\mapsto \left(13\right)
,&
3\mapsto \left(14\right)
,\\
4\mapsto \left(15\right)
,&
5\mapsto \left(16\right)
,&
6\mapsto \left(3,\,1\right)
,&
7\mapsto \left(4,\,0\right)
,\\
8\mapsto \left(4,\,2\right)
,&
9\mapsto \left(5,\,0\right)
,&
10\mapsto \left(6,\,0\right)
,&
11\mapsto \left(11,\,8\right)
,\\
12\mapsto \left(12,\,8\right)
,&
13\mapsto \left(13,\,7\right)
,&
14\mapsto \left(14,\,10\right)
,&
15\mapsto \left(17,\,9\right)
.
\end{array}}
\newcommand\subgammaII{\begin{array}{llll}
0\mapsto \left(1\right)
,&
1\mapsto \left(0\right)
,&
2\mapsto \left(8\right)
,&
3\mapsto \left(6\right)
,\\
4\mapsto \left(10\right)
,&
5\mapsto \left(9\right)
,&
6\mapsto \left(7\right)
,&
7\mapsto \left(3\right)
,\\
8\mapsto \left(5\right)
,&
9\mapsto \left(4\right)
,&
10\mapsto \left(2\right)
,&
11\mapsto \left(14\right)
,\\
12\mapsto \left(12\right)
,&
13\mapsto \left(15\right)
,&
14\mapsto \left(11\right)
,&
15\mapsto \left(13\right)
.
\end{array}}
\newcommand\subgammaOIII{\begin{array}{llll}
0\mapsto \left(14\right)
,&
1\mapsto \left(13\right)
,&
2\mapsto \left(12,\,10\right)
,&
3\mapsto \left(11,\,8\right)
,\\
4\mapsto \left(14,\,7\right)
,&
5\mapsto \left(13,\,7\right)
,&
6\mapsto \left(12,\,7\right)
,&
7\mapsto \left(\begin{array}{r}
6 \\
12
\end{array}\right)
,\\
8\mapsto \left(\begin{array}{r}
3 \\
14
\end{array}\right)
,&
9\mapsto \left(\begin{array}{r}
3 \\
13
\end{array}\right)
,&
10\mapsto \left(\begin{array}{r}
2 \\
12
\end{array}\right)
,&
11\mapsto \left(\begin{array}{rr}
6 & 1 \\
12 & 10
\end{array}\right)
,\\
12\mapsto \left(\begin{array}{rr}
6 & 1 \\
11 & 8
\end{array}\right)
,&
13\mapsto \left(\begin{array}{rr}
5 & 1 \\
15 & 9
\end{array}\right)
,&
14\mapsto \left(\begin{array}{rr}
4 & 1 \\
11 & 8
\end{array}\right)
,&
15\mapsto \left(\begin{array}{rr}
2 & 0 \\
12 & 7
\end{array}\right)
.
\end{array}}